\documentclass[11pt]{article}
\usepackage{amssymb,amsmath,dsfont}
\usepackage[latin1]{inputenc}

\def\Om     {\Omega}
\def\eps     {\varepsilon}
\def\grad    {\nabla }
\def\jnt        {\displaystyle\int}
\def\fra#1#2{\displaystyle\frac{\mathstrut #1}{\mathstrut #2}}
\def\u        {\textit{\textbf{u}}}
\def\x        {\textit{\textbf{x}}}
\def\n        {\textit{\textbf{n}}}
\def\H        {\textit{\textbf{H}}}
\def\L        {\textit{\textbf{L}}}
\def\V        {\textit{\textbf{V}}}
\def\W       {\textit{\textbf{W}}}

\newtheorem{theorem}{Theorem}
\newtheorem{lemma}[theorem]{Lemma}

\newenvironment{proof}{ \textbf{Proof}:}{\hfill}
\newtheorem{definition}[theorem]{Definition}

\begin{document}
\date{}
\title{Convergence to equilibrium of global weak solutions for a Q-tensor problem related to liquid crystals}
\author{Blanca~Climent-Ezquerra and Francisco~Guill\'en-Gonz\'alez\thanks{Departamento de Ecuaciones Diferenciales y An\'alisis Num\'erico and IMUS. Facultad de Matemáticas, Universidad de Sevilla, Spain. E-mails:{\tt bcliment@us.es, guillen@us.es}. 
Partially supported
by MINECO grant MTM2015-69875-P.}}
%
%

\maketitle

\abstract{We study a $Q$-tensor problem modeling the dynamic of nematic liquid crystals  in $3D$ domains. The system consists of the Navier-Stokes equations, with an extra stress tensor depending on the elastic forces of the liquid crystal, coupled with an Allen-Cahn system  fo                r the $Q$-tensor variable.  This problem has a dissipative in time free-energy which leads, in particular, to prove the existence of global in time weak solutions. We analyze the large-time behavior of the weak  solutions. 
By using a Lojasiewicz-Simon's result, we prove the convergence as time goes to infinity of the whole trajectory to a single equilibrium.}

\


\section{Introduction}  \label{intro}
We deal with  a generic $Q$-tensor model, following the theory of Landau-De Gennes, in a smooth and bounded domain $\Omega\subset \mathbb{R}^3$, for the
unknowns $(\u,p,Q):(0,T)\times \Omega \rightarrow \mathbb{R}^3 \times \mathbb{R} \times \mathbb{R}^{3\times 3} $, satisfying the momentum and incompressibility equations
\begin{equation}\label{modelo-u}
\left\{\begin{array}{r} D_t \u -\nu \Delta \u +\nabla p  =  \nabla \cdot \tau(Q) + \nabla \cdot \sigma (H,Q) 
\\
\noalign{\vspace{-1ex}}\\
 \nabla \cdot \u =0 
\end{array}\right.\end{equation}
and the $Q$-tensor system:
\begin{equation}\label{modelo-q}
D_t Q -S(\nabla \u, Q) = -\gamma\, {H}(Q) 
\end{equation}
in $\Omega \times (0,T).$

In \eqref{modelo-u} and (\ref{modelo-q}), $D_t=\partial_t  + (\u \cdot \nabla) $ denotes the material time derivative, $\nu>0$ is the viscosity coefficient and
$\gamma>0$ is a material-dependent elastic constant. Moreover, 
\begin{equation}\label{S-term}
S(\nabla \u , Q)= \nabla \u \, Q^t-Q^t \, \nabla \u
\end{equation}
is the so-called stretching term.

In  (\ref{modelo-u}) the tensors $\tau=\tau(Q)\in \mathbb{R}^{3\times 3}$ and
$\sigma=\sigma(H,Q)\in \mathbb{R}^{3\times 3}$ are defined by
\[
\left\{\begin{array}{rcl} \tau_{ij}(Q) &:=& -\varepsilon \,
\left( \partial_j Q :
\partial_i Q  \right)= -\varepsilon \, \partial_j Q_{kl} \, \partial_i Q_{kl} , \\
\noalign{\vspace{-2ex}}\\
\sigma (H,Q) &:=& {H} \, Q-Q \, {H},
\end{array}\right.
\]
where $ \varepsilon >0 $ and the tensor $H=H(Q)$ is related to the variational derivative in $L^2(\Omega)$ of a free
energy functional $E(Q)$, in fact
\begin{equation}\label{estrella}
E(Q):=\displaystyle\frac{\varepsilon}{2} \vert
\nabla Q \vert^2 + F(Q), 
\qquad
\mathcal{E}(Q):=\int_\Omega E(Q)\, dx,
\qquad
H:=\displaystyle\frac{\delta \mathcal{E}(Q)}{\delta Q}.
\end{equation}

Here, we denote $A:B=A_{ij} \, B_{ij}$  the scalar product of matrices (using the Einstein summation convention over repeated indices) and  the potential function $F(Q)$ is defined by
\begin{equation}\label{Fgrande}
F(Q):=\displaystyle\frac{a}{2} \, \vert Q \vert^2
-\displaystyle\frac{b}{3} \, (Q^2 : Q)
+\displaystyle\frac{c}{4} \, \vert Q \vert^4,
\end{equation}
 with $a$, $b\in \mathbb{R}$ and $c>0$. We denote by $\vert Q \vert = (Q:Q)^{1/2}$ the matrix
euclidean norm. 
Then, from (\ref{estrella}) and (\ref{Fgrande})
\begin{equation}\label{H-term}
H=H (Q)= - \varepsilon \, \Delta Q  +f(Q)
\end{equation}
 where
\[
f(Q)=\displaystyle\frac{\partial F}{\partial Q}(Q)=a \, Q - \displaystyle\frac{b}{3} \, \left(
Q^2+QQ^t+Q^tQ \right )+c  \, \vert Q \vert^2 \, Q .
\]
Finally, the system is completed with the
following initial and boundary conditions over
$\Gamma=\partial\Omega$:
\begin{equation}\label{in1}
\u\vert_{t=0} = \u_0, \qquad  Q \vert_{t=0}=Q_0 \quad \mbox{in $\Omega$,}
\end{equation}
\begin{equation}\label{bc1}
\u \vert_{\Gamma}={\bf 0}, \qquad 
 \partial_{\n}Q\vert_{\Gamma}=0 \quad \mbox{in $(0,T)$,}
\end{equation}
where $\n$ denotes the normal outwards vector on the boundary $\Gamma$.

The system (\ref{modelo-u})-(\ref{bc1}) is a simplified version of the following Q-tensor model studied by Paicu \& Zarnescu in \cite{Zarnescu} and Abels et al.~in \cite{Abels2}:
\begin{equation}\label{zm-bounded}
\left\{\begin{array}{rl} D_t \u -\nu \Delta \u +\nabla p  =  \nabla \cdot \tau(Q) + \nabla \cdot \sigma ({H}_{pz},Q) &
\mbox{in $\Omega \times
(0,T)$,}\\
\noalign{\vspace{-1ex}}\\
 \nabla \cdot \u =0 & \mbox{in
$\Omega \times (0,T)$,}\\
\noalign{\vspace{-1ex}}\\
D_t Q -(\W Q-Q\W) = -\gamma\, {H}_{pz}(Q) &
\mbox{in $\Omega \times (0,T)$,}
\end{array}\right.
\end{equation}
complemented with the initial and boundary conditions \eqref{in1}-\eqref{bc1}, 
where $\W$ is the antisymmetric part of $\nabla \u$, that is $\W:=(\nabla \u + (\nabla \u)^t)/2$, and
\[
{H}_{pz}(Q):=- \varepsilon \, \Delta Q  + a \, Q - b \, \left(
Q^2 - \displaystyle\frac{tr(Q^2)}{3} \, \mathbb{I}
\right) + c \, \vert Q \vert^2 \, Q.
\]

The model (\ref{modelo-u})-(\ref{bc1}) was studied in \cite{11},  obtaining also  the  modifications needed to assure  symmetry and traceless of $Q$. In fact, it suffices to replace  $\nabla\u$ by the antisymmetric part $\W=(\nabla\u+\nabla\u^t)/2 $ in the stretching term  $S(\nabla \u , Q)$  defined in \eqref{S-term} and the  $H(Q)$ function given in \eqref{H-term} by $H(Q)+\alpha(Q) \mathbb{I}$ where $\alpha(Q)$ is an appropriate scalar function \cite{11}. 

 These properties of symmetry and traceless are assumed (but not rigorously  justified) in \cite{Zarnescu} and  \cite{Abels2} for the model \eqref{zm-bounded}.  
Since the model (\ref{zm-bounded}) is a particular case of the general model studied in \cite{11}, then any weak solution $(\u,Q)$ of (\ref{zm-bounded}) satisfies that $Q(t)$ is a traceless and symmetric tensor.

By simplicity,  in this paper we consider the model (\ref{modelo-u})-(\ref{bc1}), because it retains  the essential difficulties of a Q-tensor model  like (\ref{zm-bounded}). In fact, the results obtained here can be extended to the Q-tensor model (\ref{zm-bounded}).

The large-time behavior of some  models for Nematic liquid crystals with unknown vector director are studied in \cite{Wu}, \cite{GrasselliWu2} (without  stretching terms) and  in \cite{LiuSun}, \cite{GrasselliWu}, \cite{WuXuLiu} (with stretching terms) and in \cite{PRS} (where different results are deduced depending on considering or not the stretching terms).

On the other hand, the large-time behavior is also analyzed  for others related models, for example in \cite{GalGrasselli} for a Cahn-Hilliard-Navier-Stokes system in $2D$ domains, in \cite{tumor} for a chemotaxis model, and  in \cite{vesicle} and \cite{nu}, where a Cahn-Hilliard-Navier-Stokes vesicle model and a smectic-A liquid crystals model are studied respectively. 

 In \cite{11bis}, some results of local in time regularity and uniqueness of the model (\ref{modelo-u})-(\ref{bc1}) are proved.

Sections 2 and 3 describe the model and the weak solution concept (more details can be seen in \cite{11}). 
The novelty of this paper is in the last two  sections. In Section~\ref{WS},  two precise  energy inequalities are proved via Galerkin Method, a time-integral version for all time $t$ and a time-differential version for almost every time. These inequalities will be essential later and they have neither been proved in \cite{PRS} nor in \cite{CavaterraRocca}.
Section~\ref{sec:time-inftys} is devoted to the study of convergence at infinite time for global weak solutions. In fact, we prove first that the $\omega$-limit set for weak solutions consists of  critical points of the free-energy. Finally, by using a  Lojasiewicz-Simon's result, we demonstrate the convergence of the whole trajectory to a single equilibrium as time goes to infinity. 
\subsection*{Notations}
The notation can be abridged. We set $L^p=L^p(\Om)$,
$p\geq 1$, $H^1=H^1(\Om)$, etc. If $X=X(\Om)$ is a space of
functions defined in the open set $\Om$, we denote by $L^p(0,T;X)$ the
Banach space $L^p(0,T;X(\Omega))$. Also, boldface letters will be used for
vectorial spaces, for instance ${\bf L}^2=L^2(\Omega)^N$, and the type ${\mathbb{L}}^2=\mathbb{L}^2(\Omega)^{N\times N}$ for the tensors. 

We set ${\mathcal V}$ the space formed by all fields $\u\in
C_0^\infty(\Om)^N$ satisfying $\nabla \u=0$. We denote $\H$
(respectively \V) the closure of ${\mathcal V} $ in $\L^2$
(respectively $\H^1$). $\H$ and $\V$ are Hilbert spaces for the
norms $\vert\cdot\vert_2$ and $\Vert\cdot\Vert_1$, respectively.
Furthermore,
\[
\H=\{\u\in \L^2; \:\nabla \u=0,\:
\u\cdot {\textbf{n}}=0 \mbox{ on }\partial\Om\},\quad
 \V=\{\u\in
\H^1;\: \nabla \u=0,\: \u=0  \mbox{ on }\partial\Om\}.
\]

From now on, $C>0$ will denote different constants, depending only
on  data of the problem.
\section{The Landau-De Gennes theory}  \label{model}
Liquid crystals can be seen as an intermediate phase of matter between crystalline solids and isotropic fluids. Nematic
liquid crystals consist of molecules with, for instance, rod-like shape whose center of mass is isotropically
distributed and whose direction is anisotropic, almost constant on average over small regions.  In the Landau-De Gennes theory, the  symmetric and traceless matrix $Q\in \mathbb{R}^{3\times 3}$, known as the Q-tensor order parameter, measures the deviation of the
second moment tensor from its isotropic value. A nematic liquid crystal is said to be isotropic when $Q = 0$, uniaxial when the Q-tensor has two equal non-zero eigenvalues and can be written in the special form:
 \[
 Q= s \left(\n\otimes \n-\frac{1}{3} \, \mathbb{I} \right) \qquad \mbox{with $s\in \mathbb{R} \backslash \{0\}$}, \, \n \in \mathbb{S}^2
 \]
and biaxial when $Q$ has three different eigenvalues and can  be represented as follows:
 \[
Q=s \, \left(\n \otimes \n -\frac{1}{3} \, \mathbb{I}\right) + r \left( {\bf m} \otimes {\bf m} -\frac{1}{3} \, \mathbb{I}\right) 
\]
where
$s, r \in \mathbb{R}; \, \n, {\bf m} \in \mathbb{S}^2.
$

The  definition of the $Q$-tensor is related to the second moment of a
probability measure $\mu(\x,\cdot): \, \mathcal{L}(\mathbb{S}^2)\rightarrow [0,1]$ for each $\x\in \Omega$, being $\mathcal{L}(\mathbb{S}^2)$ the family of Lebesgue
measurable sets on the unit sphere. For any $A\subset \mathbb{S}^2$, $\mu(\x,A)$ is the probability that
the molecules with centre of mass in a very small neighborhood of the point $\x \in \Omega$ are pointing in a direction
contained in $A$.
This probability must satisfy $\mu(\x,A)=\mu(\x,-A)$ in order to reproduce the so-called
``head-to-tail'' symmetry.
As a consequence, the first moment of the probability measure vanishes, that is
\[
\langle p \rangle (\x)=\displaystyle\int_{\mathbb{S}^2} p_i \, d\mu(\x,p)=0.
\]
Then, the main information on $\mu$ comes from the second moment tensor
\[
M(\mu)_{ij}=\displaystyle\int_{\mathbb{S}^2} p_i \, p_j \,
d\mu(p),\quad i,j=1,2,3.
\]
 It is easy to see that
$M(\mu)=M(\mu)^t$ and $tr(M)=1$. If the orientation of the
molecules is equally distributed, then the distribution is
isotropic and $\mu=\mu_0$, $d\mu_0(p)=\frac{1}{4\pi} \, dA$
and $M(\mu_0)=\frac{1}{3} \, \mathbb{I}$.  The deviation of the
second moment tensor from its isotropic value is therefore measured as:
\[
Q=M(\mu)-M(\mu_0)=\displaystyle\int_{\mathbb{S}^2}
\left(p\otimes p - \displaystyle\frac{1}{3}\,
\mathbb{I}\right) \, d\mu(p).
\]
From this equality, $Q$ is  symmetric and traceless. 
\section{Weak solutions} \label{se:weak}

\begin{definition}[Weak solution]\label{dweak}
It will be said that $(\u,Q)$ is a weak solution in $(0,+\infty)$
of  problem (\ref{modelo-u})-(\ref{bc1}) if:
\begin{equation}\label{wr}
\left\{\begin{array}{l}
\u  \in  L^{\infty}(0,+\infty;{\bf H}) \cap L^2(0,+\infty;{\bf V}), \\
[2mm]
Q  \in  L^{\infty}(0,+\infty;\mathbb{H}^1(\Omega)) \cap
L^2(0;T;\mathbb{H}^2(\Omega)) \quad \forall T>0, 
\end{array}\right.
\end{equation}
and satisfies the variational formulation (\ref{cero}) and
(\ref{doblecero}) (defined below), the initial conditions (\ref{in1}) and the boundary conditions (\ref{bc1}). 
\end{definition}

Note that the regularity imposed in (\ref{wr}) is satisfied up to infinite time excepting the  $\mathbb{H}^2(\Omega)$-regularity for $Q$.

In  \cite{11} the following result is proved by means of a Galerkin approximation.
\begin{theorem}[Existence of weak solutions] \label{Th:weak}
If  $(\u_0,Q_0)\in \H\times \mathbb{H}^1(\Omega)$,  there exists a weak solution $(\u,Q)$ of  system
(\ref{modelo-u})-(\ref{bc1}) in $(0,+\infty)$.
\end{theorem}

\subsection*{Variational formulation}\label{vf1}
Taking into account that $\partial_i F(Q)=F'(Q): \partial_i
Q=f(Q): \partial_i Q$, the term of the symmetric tensor $\tau(Q)$ can be
rewritten as:
\[
(\nabla \cdot \tau(Q))_i  =  
H(Q): \partial_i Q - \partial_i \left(F(Q)
+\displaystyle\frac{\varepsilon}{2} \,  \vert \nabla Q
\vert^2  \right) ,
\]
where $\vert \nabla Q \vert^2=\partial_j Q : \partial_j Q $. Then, testing (\ref{modelo-u}) by any  $\widetilde{\u}:\Omega\to \mathbb{R}^3$ with $\widetilde{\u}\vert_{\partial\Omega} ={\bf 0}$ and $\nabla \cdot\widetilde{\u} =0$ in $\Omega$, we
arrive at the following variational formulation of (\ref{modelo-u}):
\begin{equation}\label{cero}
\begin{array}{c}
 (D_t\u,\widetilde{\u})  + \nu (\nabla \u, \nabla \widetilde{\u})
- ((\widetilde{\u} \cdot \nabla)Q, H)
+ (\sigma(H,Q),\nabla \widetilde{\u})=0.
 \end{array}
\end{equation}

On the other hand, testing (\ref{modelo-q}) by any $\widetilde{H}$ and the system
$-\varepsilon \, \Delta Q +f(Q)=H$ by any $\widetilde{Q}$, we arrive
at the variational formulation:
\begin{equation}\label{doblecero}
\left\{\begin{array}{l} (\partial_t Q,\widetilde{H}) + ((\u \cdot \nabla ) Q, \widetilde{H})
-(S(\nabla \u,Q),\widetilde{H})+ \gamma \, (H, \widetilde{H})  = 0 ,\\[3mm]
\varepsilon \, (\nabla Q, \nabla
\widetilde{Q})+(f(Q),\widetilde{Q})-(H,\widetilde{Q})=0,
\end{array}\right.
\end{equation}
for any $\widetilde{H}$, $\widetilde{Q}:\Omega\to \mathbb{R}^{3\times 3}$.
From (\ref{doblecero}), one has in particular:
\begin{equation} \label{doblecero-a}
\begin{array}{c}
(\partial_t Q, \widetilde{Q})+(
(\u \cdot \nabla) Q, \widetilde{Q}
)-(S(\nabla \u,Q),\widetilde{H}) 
-\varepsilon \, \gamma \, (\Delta Q,\widetilde{Q}) + \gamma \, (
f(Q),\widetilde{Q}
)=0.
\end{array}
\end{equation}
On the other hand, by applying regularity \eqref{wr} to the systems \eqref{cero} and \eqref{doblecero-a}, one has 
\[
\partial_t\u \in L^{4/3}_{loc}([0,+\infty);{\bf V}')\quad\mbox{and}\quad
\partial_t Q \in L^{4/3}_{loc}([0,+\infty);\mathbb{L}^{2}(\Omega)),
\]
hence, the following time-continuity can be deduced:
\[
\u\in C([0,+\infty);{\bf V}')\cap C_w([0,+\infty);{\bf H}),\ 
Q\in C([0,+\infty);\mathbb{L}^{2}(\Omega))\cap C_w([0,+\infty);\mathbb{H}^{1}).
\]

In particular, the initial conditions (\ref{in1}) have sense.
\subsection*{Dissipative energy law and global in time a priori estimates}\label{ss-le}

Now, we argue in a formal manner, assuming a regular enough solution $(\u,p,Q)$ of 
(\ref{modelo-u})-(\ref{bc1}).

By taking $\widetilde{\u}=\u$ in (\ref{cero}) and $(\widetilde{H},\widetilde{Q})=(H,\partial_t Q)$ in (\ref{doblecero}) 
then the stretching term cancels with the term dependent on the tensor $\sigma(H,Q)$, the term $((\u \cdot \nabla)Q, H)$ appearing in both (\ref{cero}) and (\ref{doblecero}) also cancel and the convection term $((\u \cdot\nabla) \u,\u)$ vanishes, hence the following ``energy equality" holds:
\begin{equation}\label{leyenergiapre}
\frac{d}{dt} \left(
\displaystyle\frac{1}{2} \Vert \u \Vert_{{\bf L}^2(\Omega)}^2 + \displaystyle\int_{\Omega}E(Q) \, d\x
\right) + \nu \Vert \nabla \u \Vert_{\mathbb{L}^2}^2 + \gamma \Vert {H} \Vert_{\mathbb{L}^2}^2 = 0 .
\end{equation}

For the moment, bounds for $(\u,Q)$ are not guaranteed from (\ref{leyenergia}) because $ \displaystyle\int_{\Omega}
E(Q) \, d\x$ is not a positive term due to $F(Q)$. However,
it is possible to find a large enough constant $\mu>0$ depending on parameters $a$, $b$ and $c$ given in the  definition of $F(Q)$ in (\ref{Fgrande}),  such that 
\begin{equation}\label{estrella11}
F_{\mu}(Q):=F(Q)+\mu \ge \displaystyle\frac{c}{8} \,
\vert Q \vert^4.
\end{equation}
By replacing $E(Q)$  in (\ref{leyenergiapre})  by
\[
E_{\mu}(Q):=\displaystyle\frac{1}{2} \vert \nabla
Q \vert^2 + F_{\mu}(Q)  \ge 0,
\]
and denoting the kinetic and phase energies as
 \[
 \mathcal{E}_k(\u(t)):=\fra{1}{2}\Vert\u\Vert_{{\bf L}^2}^2\quad \mbox{ and }\quad
\mathcal{E}_\mu(Q):=\displaystyle\int_{\Omega} E_\mu(Q) \, d\x
\]
and the total energy as 
 \[\mathcal{E}(\u,Q):=\mathcal{E}_k(\u) +  \mathcal{E}_\mu(Q) ,
 \]
 then (\ref{leyenergiapre}) implies 
\begin{equation}\label{leyenergia}
\frac{d}{dt}\mathcal{E}(\u(t),Q(t)) + \nu \Vert \nabla \u \Vert_{{\bf L}^2}^2 + \gamma \Vert {H} \Vert_{\mathbb{L}^2}^2 = 0 .
\end{equation}
This energy equality shows the dissipative character of the model with respect to the total free-energy $\mathcal{E}(\u(t),Q(t))$. 
In fact, assuming finite total energy of initial data, i.e.
\[
 \displaystyle\int_{\Omega} E_{\mu}(Q_0) \, d\x +
\displaystyle\frac{1}{2} \Vert \u_0 \Vert_{{\bf L}^2(\Omega)}^2 < +\infty ,
\]
 then the following estimates hold:
\begin{equation}\label{reg-deb}
\begin{array}{c} 
\u  \in   L^{\infty}(0,+\infty;\L^2(\Omega))  \cap L^2(0,+\infty;{\bf {H}}^1(\Omega)),\\[3mm]
\nabla Q  \in  L^{\infty}(0,+\infty;\mathbb{L}^2(\Omega)),
\quad
{F}_{\mu}(Q) \in  L^{\infty}(0,+\infty;L^1(\Omega)),
\\[3mm]
{H} \in  L^2(0,+\infty;\mathbb{L}^2(\Omega)).
\end{array}
\end{equation}
In particular, from (\ref{estrella11}) and (\ref{reg-deb}),
we deduce the regularity:
\[
\begin{array}{l}
Q \in  L^{\infty}(0,+\infty; \mathbb{L}^4(\Omega)) \quad \mbox{and} \quad Q \in
L^{\infty}(0,+\infty;\mathbb{H}^1(\Omega)),
\end{array}
\]
hence, in particular
\begin{equation}\label{estrella2}
Q \in L^{\infty}(0,+\infty;\mathbb{L}^6(\Omega)).
\end{equation}
Since $f(Q)$ is  a third order polynomial function,
\[
\vert {f}(Q)\vert \le C(a,b,c) \, \left( \vert Q \vert + \vert Q \vert^2 +
\vert Q \vert^3 \right)
\]
which, together with (\ref{estrella2}), gives ${f}(Q)\in L^{\infty}(0,+\infty;\mathbb{L}^2(\Omega))$.
Then, using that ${H}(Q)=-\varepsilon \, \Delta Q +
{f}(Q)$, we obtain:
\[
\Delta Q \in
L^{\infty}(0+\infty;\mathbb{L}^2(\Omega)) +
L^2(0,+\infty;\mathbb{L}^2(\Omega))\]
hence
\[ 
\Delta Q \in
L^2(0,T;\mathbb{L}^2(\Omega)) \quad \forall T>0.
\]
Finally, by using the $H^2$-regularity of the Poisson problem:
\[
\left\{\begin{array}{rcl}
- \varepsilon\, \Delta Q + Q &=& f(Q)+Q \quad \mbox{in $\Omega$,}\\[1mm]
\partial_\n Q \vert_{\Gamma} &=& {0} 
\end{array}\right.
\]
we deduce that:
\[
 Q \in
L^2(0,T;\mathbb{H}^2(\Omega)) \quad \forall \,T>0.
\]
\section{Two improved energy inequalities}\label{WS}
Now, we are in order to prove the following technical lemma.
 \begin{lemma} \label{weakdef} 
 Let $(\u,Q)$ be a weak solution  in $(0,+\infty)$
of  problem (\ref{modelo-u})-(\ref{bc1}) furnished by a 
 Galerkin approximation. Then, $(\u,Q)$ satisfies the following energy inequality a.e.~$t_1, t_0: t_1\geq t_0\geq 0$:
  \begin{equation}\label{energyeqint0}
\mathcal{E}(\u(t_1),Q(t_1))-\mathcal{E}(\u(t_0),Q(t_0))
+ \jnt_{t_0}^{t_1}(\nu\|\grad{\u}(s)\|^2_{\mathbb{L}^2} +
\gamma\|H(s)\|^2_{\mathbb{L}^2})\;ds  \leq 0.
\end{equation}

 Moreover, there exists a special function $\widetilde{\mathcal{E}}=\widetilde{\mathcal{E}} (t)\in \mathbb{R}$ defined for all $t\geq 0$, which satisfies the following integral inequality for all $t_1, t_0: t_1\geq t_0\geq 0$:
 \begin{equation}\label{energyeqint}
\widetilde{\mathcal{E}}(t_1)-\widetilde{\mathcal{E}}(t_0)
+ \jnt_{t_0}^{t_1}(\nu\|\grad{\u}(s)\|^2_{\mathbb{L}^2} +
\gamma\|H(s)\|^2_{\mathbb{L}^2})\;ds  \leq 0,
\end{equation}
  and the following differential version a.e. $t\ge 0$:
 \begin{equation}\label{energy-eqs}
  \qquad\displaystyle \frac{d}{dt}\widetilde{\mathcal{E}}(t)+ \nu\|\grad{\u}(t)\|^2_{\mathbb{L}^2} +\gamma\|H(t)\|^2_{\mathbb{L}^2}  \leq 0.
  \end{equation}
 \end{lemma}

\begin{proof}
 To prove (\ref{energyeqint}) we start from the following energy equality satisfied by the Galerkin approximate solutions (see \cite{11}) for all $t, t_0$ with $t\geq t_0\ge 0$:
\begin{equation}\label{energyeqintm}
\mathcal{E}(\u_m(t),Q_m(t))-\mathcal{E}(\u_m(t_0),Q_m(t_0))
+ \jnt_{t_0}^{t}(\nu\|\grad{\u_m}(s)\|^2_{L^2} +
\gamma\|H_m(s)\|^2_{\mathbb{L}^2})\;ds  \leq 0.
\end{equation}
 Moreover, $\u_m(t)$ and $Q_m(t)$ have sufficient estimates to obtain  
\begin{equation}\label{en1}
\begin{array}{c}
\displaystyle\mathcal{E}(\u_m(t),Q_m(t))\rightarrow \mathcal{E}(\u(t),Q(t))\quad 
\hbox{in $L^1(0,T)$, and in particular a.e. $t\geq 0$}.
\end{array}
\end{equation}
Since $ \u_m\to\u$ weakly in $L^2(0,T;\H^1)$ and $ H_m\to H$ weakly in $L^2(0,T;\mathbb{L}^2)$, 
\begin{equation}\label{en4}
\displaystyle\liminf_{m\rightarrow +\infty}\jnt_{t_0}^{t_1}(\nu\|\grad{\u_m}(s)\|^2_{\mathbb{L}^2} +\gamma\|H_m(s)\|^2_{\mathbb{L}^2})\;ds   
\geq
\jnt_{t_0}^{t_1}(\nu\|\grad{\u}(s)\|^2_{\mathbb{L}^2} +
\gamma\|H(s)\|^2_{\mathbb{L}^2})\;ds 
\end{equation}
for all $t_1, t_0: t_1\geq t_0\ge 0$.

By taking $\liminf_{m\rightarrow +\infty}$ in (\ref{energyeqintm}), we obtain that for all $t_1\geq t_0\geq 0$,
\begin{equation}\label{energlim}
\begin{array}{c}
\displaystyle\liminf_{m\rightarrow +\infty}\mathcal{E}(\u_m(t),Q_m(t))+
\displaystyle\liminf_{m\rightarrow +\infty} \jnt_{t_0}^{t_1}(\nu\|\grad{\u_m}(s)\|^2_{L^2} +
\gamma\|H_m(s)\|^2_{\mathbb{L}^2})\;ds 
\\[4mm]\leq 
\displaystyle \limsup_{m\rightarrow +\infty}\mathcal{E}(\u_m(t_0),Q_m(t_0)).
\end{array}
\end{equation}
By using (\ref{en1}) and (\ref{en4}) in (\ref{energlim}), we obtain \eqref{energyeqint0}.


 
 On the other hand, since the inequality (\ref{energyeqint0}) is satisfied for all $t_0, t_1\in [0,+\infty)\backslash N$, where $N$ is a set of null Lebesgue measure, then the map $t\in  [0,+\infty)\backslash N\to \mathcal{E}(\u(t),Q(t))\in \mathbb{R}$ is a real decreasing (and bounded) function. The, we can define a special  function $\widetilde{\mathcal{E}}(t)$ for all $t\in [0,+\infty)$ as:
\[
\widetilde{\mathcal{E}}(0):=\mathcal{E}(u_0,Q_0), \qquad \widetilde{\mathcal{E}}(t):=\lim_{\stackrel{\:\:\scriptstyle s\to t^-} {\scriptstyle s \in [0,+\infty)\setminus N}}\mathcal{E}(\u(s),Q(s)).
\]
This function $\widetilde{\mathcal{E}}$ is ``continuous from the left"  and decreasing for all $t\ge 0$. Indeed, for any $t_1, t_2\in  [0,+\infty)$, for instance $t_1< t_2$, we can choose sequences $\{s^1_n\}, \{s^2_n\} \subset [0,+\infty)\backslash N$ such that $s_n^1\to t^-_1$, $s_n^2\to t^-_2$ and,   $s^1_n \leq s^2_n$ for all  $n\geq n_0$. 
Since $s_n^1$ and  $s_n^2$ are not in $N$, we know that $\mathcal{E}(\u(s_n^1),Q(s_n^1))\geq \mathcal{E}(\u(s_n^2),Q(s_n^2))$. By taking limit as $s_n^1\to t^-_1$ and $s_n^2\to t^-_2$, we obtain that $\widetilde{\mathcal{E}}(t_1)\geq\widetilde{\mathcal{E}}(t_2)$.

Since $\widetilde{\mathcal{E}}(t)$ is decreasing for all $t\in [0,+\infty)$, it is derivable (and absolutely continuous) almost everywhere $t\in (0,+\infty)$.

Since the inequality (\ref{energyeqint0}) is satisfied for all $t_0, t_1 \in [0,+\infty)\setminus N$ where the measure of $N$  is zero, given any $t_0<t_1$, we can take $\delta_n>0$ and $\eta_n>0$ such that $t_0-\delta_n,\:t_1-\eta_n \not\in N$ and $\delta_n,\eta_n\to 0$, hence 
\[
\widetilde{\mathcal{E}}(t_1-\eta_n)-\widetilde{\mathcal{E}}(t_0-\delta_n)
+ \jnt_{t_0-\delta_n}^{t_1-\eta_n}(\nu\Vert\grad{\u}(s)\Vert^2_{L^2} +
\gamma\Vert\grad H(s)\Vert^2_{L^2})\;ds  \leq 0.
\]
By taking $\delta_n\to 0$ and $\eta_n\to 0$, we obtain (\ref{energyeqint}).

In particular,   
by choosing $t_0=t$ and $t_1=t+h$ in (\ref{energyeqint}), we obtain
 \begin{equation}\label{interm}
\fra{\widetilde{\mathcal{E}}(t+h)-\widetilde{\mathcal{E}}(t)}{h} +\fra{1}{h} \jnt_{t}^{t+h}(\nu\Vert\grad{\u}(s)\Vert^2_{L^2} +
\gamma\Vert\grad H(s)\Vert^2_{L^2})\;ds  \leq 0,\quad \forall\, t,h\ge 0.
\end{equation}
Observe that
\[
\lim_{h\to 0}\fra{1}{h} \jnt_{t}^{t+h}(\nu\Vert\grad{\u}(s)\Vert^2_{L^2} +
\gamma\Vert\grad H(s)\Vert^2_{L^2})\;ds  = \nu\Vert\grad{\u}(t)\Vert^2_{L^2} +
\gamma\Vert\grad H(t)\Vert^2_{L^2},
\]
a.e.~$t\ge 0$
because the map, $s\in [0,+\infty)\to \nu\Vert\grad{\u}(s)\Vert^2_{L^2} +
\gamma\Vert\grad H(s)\Vert^2_{L^2}\in \mathbb{R}$, belongs to $L^1(0,+\infty)$. 
Accordingly, by taking $h\to 0$ in (\ref{interm}), we obtain (\ref{energy-eqs}) a.e.~$t\ge 0$.
\hfill $\square$
\end{proof}
\section{Convergence at infinite time. } \label{sec:time-inftys}
Let $(\u,Q)$ be  a weak solution of (\ref{modelo-u})-(\ref{bc1}) in $(0,+\infty)$ associated to an  initial data $(\u_0,Q_0)\in \H\times \mathbb{H}^1(\Omega)$  (see Definition~\ref{dweak}) satisfying Lemma~\ref{weakdef}.  
From the energy inequality (\ref{energyeqint0}), there exists a real number $E_\infty\geq 0$ such that the total energy evaluated in the trajectory $(\u(t),Q(t))$ satisfies 
\begin{equation}\label{asenergys}
\mathcal{E}(\u(t),Q(t))\searrow E_\infty\mbox{ in }\mathbb{R} \quad\mbox{ as } t\uparrow +\infty.
\end{equation}
Let us define the  $\omega$-limit set of  this  global weak solution $(\u,Q)$  as follows:
\[ 
\begin{array}{l}
\omega(\u,Q)=\{(\u_\infty,Q_\infty)\in \H\times \mathbb{H}^1: \exists \{t_n\}\uparrow+\infty \mbox{ s.t.\ }
\\[3mm]\quad
({\u}(t_n),{Q}(t_n))\rightarrow(\u_\infty,Q_\infty) \mbox{ weakly in } \L^2\times \mathbb{H}^1\}.
 \end{array}
\] 
Let ${\cal S}$ be the set of critical points of the energy $\mathcal{E}(Q)$ defined in (\ref{estrella}), that is
\[ 
\quad{\cal S}=\{Q\in \mathbb{H}^2:  -\eps \Delta Q+f(Q)=0 \mbox{ in }\Omega, \:\partial_n Q\vert_{\Gamma}=0  \}.
\] 
\begin{theorem} \label{the:first}
Assume that $(\u_0,Q_0)\in \H\times \mathbb{H}^1$. Fixed $(\u,Q)$  a weak solution of (\ref{modelo-u})-(\ref{bc1}) in $(0,+\infty)$ satisfying Lemma~\ref{weakdef},
then $\omega(\u,Q)$ is  nonempty  and $\omega(\u,Q)\subset \{0\}\times {\cal S}$. Moreover, for any  $ Q_\infty \in {\cal S}$ such that $(0,Q_\infty)\in \omega(\u,Q)$, it holds  
\[
\mathcal{E}_\mu(Q_\infty)=E_\infty.
\]
 In particular, $\u(t)\rightarrow 0$ weakly in $\L^2$ and $\mathcal{E}_\mu(Q(t))\rightarrow\mathcal{E}_\mu(Q_\infty)$ in $\mathbb{R}$ as $t\uparrow +\infty$.
\end{theorem}
{\bf Proof: } Observe that since 
\[
(\u,Q)\in L^\infty(0,+\infty;\H\times \mathbb{H}^1),
\]
 for any sequence $\{t_n\}\uparrow +\infty$ there exists a subsequence (equally denoted) and suitable limit functions $(\u_\infty, Q_\infty)\in \H\times \mathbb{H}^1$, such that 
\begin{equation}\label{phiinf}
\quad\:
\u(t_n)\rightarrow \u_\infty\mbox{ weakly in }\H,\:
Q(t_n)\rightarrow Q_\infty\mbox{ weakly in }\mathbb{H}^1 .
\end{equation}
We consider the initial and boundary-value problem associated to (\ref{modelo-u})-(\ref{bc1}) restricted on the time interval $[t_n,t_n+1]$ with initial values $\u(t_n)$ and $Q(t_n)$. If we define 
\[
\u_n(s):=\u(s+t_n),\: Q_n(s):=Q(s+t_n),\: H_n(s):=H(s+t_n)
\]
 for a.e. $s\in [0,1]$,
then, $(\u_n, Q_n)$ is a weak solution to the problem (\ref{modelo-u})-(\ref{bc1}) in the time interval $ [0,1]$.
From the energy inequality (\ref{energyeqint0}), we have that
\[
\begin{array}{c}
 \jnt_0^1( \nu \Vert \nabla \u_n(s) \Vert_{{\bf L}^2}^2 + \gamma \Vert {H}_n(s) \Vert_{\mathbb{L}^2}^2 )\;ds 
=\jnt_{t_n}^{t_{n}+1}(\nu|\Vert \nabla \u(t) \Vert_{{\bf L}^2}^2+
\gamma\Vert H(t)\Vert^2_{\mathbb{L}^2})\;dt
\\[0.3cm]
 \leq\mathcal{E}_\mu(Q(t_n))-\mathcal{E}_\mu(Q(t_n+1))\longrightarrow 0\quad \mbox{ as }n\rightarrow \infty,
\end{array} 
\]
 hence, 
\[
\grad \u_n\rightarrow 0\mbox{ strongly in }L^2(0,1;\L^2) 
\]
 and 
 \[H_n\rightarrow 0\mbox{ strongly in }L^2(0,1;\mathbb{L}^2).
 \]
  In particular, by using Poincar\'e inequality, one has
\[
\u_n\rightarrow 0\:\mbox{ strongly in }L^2(0,1;\V)
\]
and
\[H_n\rightarrow 0\:\mbox{ strongly in }L^2(0,1;\mathbb{L}^2). 
 \]
 Moreover, since $\u_n$ and $\partial_t \u_n$ are bounded in $L^{\infty}(0,1;\H)$ and $L^{4/3}(0,1;\V ')$ respectively, then 
 $\u_n\rightarrow 0$ in  $C([0,1];\V')$. In particular, $\u(t_n)=\u_n(0)\rightarrow0$ in  $\V '$, hence  $\u_\infty=0$ (owing to (\ref{phiinf})). Consequently, the whole trajectory $\u(t)\to 0$ as $t\to +\infty$.

Furthermore,  
$Q_n$  is bounded in $L^2(0,1;\mathbb{H}^2) \bigcap L^\infty(0,1;\mathbb{H}^1)$ and  $\partial_t Q_n$ is bounded in \\$L^{4/3}(0,1;\mathbb{L}^2) $. Therefore,  there exists a subsequence of $Q_n$ (equally denoted) and a limit function $\overline{Q}$ such that
$
 Q_n\rightarrow \overline{Q}$ strongly in $C^0([0,1]\times \overline\Om)\cap L^2(0,1;\mathbb{H}^1)
 $ and weakly in $L^2(0,1;\mathbb{H}^2)$.

In particular, $Q(t_n)=Q_n(0)\rightarrow \overline{Q}(0)$ in $C^0(\overline\Om)$, hence  $\overline{Q}(0)=Q_\infty$ (owing to (\ref{phiinf})). 
On the other hand,  $\partial_t Q_n$ converges weakly to $\partial_t \overline Q$ in $L^{4/3}(0,1;\mathbb{L}^2)$, hence taking limits in the variational formulation:
\[
\begin{array}{r}
(\partial_t Q_n, \widetilde{Q})+(
(\u_n \cdot \nabla) Q_n, \widetilde{Q}
) 
-(S(\nabla\u_n,Q_n),\widetilde{Q})\\\quad
-\varepsilon \, \gamma \, (\Delta Q_n,\widetilde{Q}) + \gamma \, (
f(Q_n),\widetilde{Q}
)=0.
\end{array}
\]
for all $\widetilde{Q}\in \mathbb{L}^2$, we have that 
$\partial_tQ_n\rightarrow 0$ in $L^{4/3}(0,1;\mathbb{L}^2)$. Therefore, 
 $\partial_t \overline Q= 0$ and $\overline Q(t)$ is a constant function of $\mathbb{H}^1$ for all $t \in [0,1]$, hence since $\overline{Q}(0)=Q_\infty$, we have 
\begin{equation}\label{barinfty}
\overline Q(t)=Q_\infty\in \mathbb{H}^1\quad\mbox{ for all } t\in [0,1].
\end{equation}
Finally, since $f(Q_n)$ converges weakly in $L^{\infty}(0,1;\mathbb{L}^2)$,
by taking limit as $n\rightarrow +\infty$ in the variational formulation  
$(H_n,\widetilde{Q})=\varepsilon \, (\nabla Q_n, \nabla
\widetilde{Q})+(f(Q_n),\widetilde{Q})$ for all $\widetilde{Q}\in \mathbb{H}^1$, 
 we deduce 
\[
\varepsilon \, (\nabla \overline{Q}, \nabla
\widetilde{Q})+(f(\overline{Q}),\widetilde{Q})=0, \quad \forall\,\widetilde{Q}\in \mathbb{H}^1, \mbox{ a.e. }t\in (0,1).
\]
Then, from (\ref{barinfty}),  $Q_\infty\in \mathbb{H}^1$ and $\varepsilon \, (\nabla Q_\infty, \nabla
\widetilde{Q})+(f(Q_\infty),\widetilde{Q})=0, \: \forall\,\widetilde{Q}\in \mathbb{H}^1, \mbox{ a.e. }t\in (0,1)$. 
Finally, by applying $\mathbb{H}^2$-regularity of the Poisson problem:
\[
\left\{\begin{array}{rcl}
- \varepsilon\, \Delta Q + Q &=& f(Q)+Q \quad \mbox{in $\Omega$,}\\[1mm]
\partial_\n Q \vert_{\Gamma} &=& {0} 
\end{array}\right.
\]
we deduce that $Q_\infty\in \mathbb{H}^2$, hence $Q_\infty\in {\cal S}$ and the proof is finished. 
\hfill $\square$

\

In the next theorem we apply the following Lojasiewicz-Simon's result that can be found in \cite{PRS}.
\begin{lemma}[Lojasiewicz-Simon inequality] \label{le:L-S2} 
 Let  ${Q}_*\in {\cal S}$ and $K>0$ fixed. Then, there exists positive constants $\beta_1$, $\beta_2$ and  $C$ and $\theta\in(0,1/2]$, 
 such that for all $Q\in \mathbb{H}^2$ with  $\Vert Q\Vert_{\mathbb{H}^1}\leq K$,  $\Vert Q-{Q}_*\Vert_{\mathbb{L}^2}\leq \beta_1$ and $\vert \mathcal{E}(Q)-\mathcal{E}({Q}_*)\vert\leq\beta_2$, it holds 
\[ 
 \vert \mathcal{E}(Q)-\mathcal{E}({Q}_*)\vert^{1-\theta}\leq C\,\Vert H\Vert_{\mathbb{H}^{-1}}
\] 
 where $H=H(Q)$ is defined in (\ref{doblecero}).
 \end{lemma}
\begin{theorem} \label{the:second}
Assume that $\widetilde{\mathcal{E}}(t)$ belongs to the equivalence class of the energy function \\$\mathcal{E}(\u(t),Q(t))$, that is, 
$\widetilde{\mathcal{E}}(t)={\mathcal{E}}(\u(t),Q(t))$ almost everywhere $t\geq 0$.
Then, under the hypotheses of  Theorem~\ref{the:first}, there exists a unique limit $Q_\infty\in {\cal S}$ such that 
 $Q(t)\rightarrow Q_\infty$ in $\mathbb{H}^1$-weakly as $t\uparrow +\infty$, i.e.  $\omega(\u,Q)=\{(0,Q_\infty)\}$. 
\end{theorem}
%
%
%
%
{\bf Proof: } 
Let $Q_\infty\in{\cal S} $ such that $(0,Q_\infty)\in \omega(\u,Q)$, i.e.~there exists $t_n\uparrow +\infty$ such that $\u(t_n)\rightarrow 0$ weakly in $\L^2$ and $Q(t_n)\rightarrow Q_\infty$ weakly in  $\mathbb{H}^1$ (and strongly in $\mathbb{L}^2$).
\medskip

Without loss of generality, it can be assumed that $\widetilde{\mathcal{E}}(t)> {\mathcal{E}}_\mu(Q_\infty)(=E_\infty)$ for all $t> 0$, because otherwise, if it exists some $\widetilde{t}>0$ such that $\widetilde{\mathcal{E}}(\widetilde{t})= E_\infty$, then the energy inequality $(\ref{energyeqint})$   implies
\[
\widetilde{\mathcal{E}}(t)= E_\infty,\quad \forall\, t\geq\widetilde{t},
\]
\[
 \Vert \nabla \u(t) \Vert_{\mathbb{L}^2}^2 =0 \quad \hbox{and} \quad \Vert {H}(t) \Vert_{\mathbb{L}^2}^2=0, \quad \forall\, t\geq\widetilde{t}.
\]
 Therefore, $\u(t)=0$ and $H(t)=0$ for all $t\geq\widetilde{t}$, and  by using the $Q$-equation (\ref{modelo-q}), $\partial_t
 Q(t)=0$,  hence $Q(t)=Q_\infty$ for all $t\geq\widetilde{t}$. In this setting the convergence of the whole $Q$-trajectory towards $Q_\infty$ is trivial.

Therefore, we can assume that $\widetilde{\mathcal{E}}(t)> E_\infty$ for all $t\geq 0$. Then, the proof will be  divided into three steps.
\medskip

\noindent{\bf Step 1:} {\sl Assuming that there exists $t_1>0$ such that 
\[
\Vert Q(t)-Q_\infty\Vert_{\mathbb{L}^2}\leq\beta_1\:\mbox{ and }\:\vert {\mathcal{E}}_\mu(Q(t))-{\mathcal{E}}_\mu({Q}_\infty)\vert\leq\beta_2
\]
 for all $ t\geq t_1\geq 0$, where $\beta_1>0, \beta_2>0$ are the constants appearing in Lemma~\ref{le:L-S2} (of Lojasiewicz-Simon's type), then the following inequalities hold:
\medskip
\begin{equation}\label{stabs}
\begin{array}{l}
\fra{d}{dt}\Big( (\widetilde{\mathcal{E}}(t)-E_\infty)^\theta\Big)
+{C}\,{\theta}\:(\Vert \grad\u(t)\Vert_{{\bf L}^2}+
\Vert H(t)\Vert_{\mathbb{L}^2})\leq 0,
\end{array}
\end{equation} 
$\mbox{a.e.}~t\in (t_1,\infty)
$.
\begin{equation}\label{stab2s}
\jnt_{t_{1}}^{t_2}\Vert\partial_t Q\Vert_{\mathbb{H}^{-1}}\leq\fra{C}{\theta}(\widetilde{\mathcal{E}}(t_{1})-E_\infty)^\theta,
\end{equation}
 for all $t_2\in (t_1,\infty)$, 
where $\theta\in (0,1/2]$ is the  constant appearing in Lemma~\ref{le:L-S2}. 
}

Since $E_\infty$ is constant, we can rewrite  the energy inequality (\ref{energy-eqs}) as 
\[
  \fra{d}{dt} (\widetilde{\mathcal{E}}(t)-E_\infty) + C\left( \Vert \nabla \u(t) \Vert_{\mathbb{L}^2}^2 + \Vert {H} (t)\Vert_{\mathbb{L}^2}^2 \right)\leq 0,
  \] 
 almost everywhere $t\ge 0.$
By taking into account that 
\[
 \displaystyle \Vert \nabla{\u}(t)\Vert_{\mathbb{L}^2}^2  +
\Vert H(t)\Vert_{\mathbb{ L}^2}^2\geq  \frac{1}{2}\left(\Vert \nabla{\u}(t)\Vert_{\mathbb{L}^2} +
\Vert H(t)\Vert_{\mathbb{ L}^2}\right)^2
\]
 and the inequality 
\[
\displaystyle  \frac{1}{2} (\Vert\grad\u(t)\Vert_{\mathbb{L}^2}+\Vert H(t)\Vert_{\mathbb{ L}^2}) \ge C(\Vert\u(t)\Vert_{{\bf L}^2}+\Vert H(t)\Vert_{\mathbb{H}^{-1}}),
\]
 we obtain
\[ 
\fra{d}{dt} (\widetilde{\mathcal{E}}(t)-E_\infty) + C (\Vert\u(t)\Vert_{{\bf L}^2}+\Vert H(t)\Vert_{\mathbb{H}^{-1}})
\left( \Vert \nabla \u(t) \Vert_{\mathbb{L}^2} + \Vert {H} (t)\Vert_{\mathbb{L}^2} \right)
\leq 0, \quad\mbox{ a.e.}~t\ge 0.
\]
By using this expression and the time derivative of the (strictly positive) function 
$(\widetilde{\mathcal{E}}(t)-E_\infty)^\theta ,
$
 we obtain a.e.~$t\ge 0$ that
\begin{equation}\label{estim-1s}
\begin{array}{l}
 \fra{d}{dt} \left((\widetilde{\mathcal{E}}(t)-E_\infty)^\theta\right)
  \\
 \qquad  
 + \theta(\widetilde{\mathcal{E}}(t)-E_\infty)^{\theta-1}
 C
  (\Vert\u(t)\Vert_{{\bf L}^2}+\Vert H(t)\Vert_{\mathbb{H}^{-1}})
\left( \Vert \nabla \u(t) \Vert_{\mathbb{L}^2} + \Vert {H} (t)\Vert_{\mathbb{L}^2} \right)
\leq 0.
\end{array}
\end{equation}
On the other hand, 
by taking into account that 
$\vert {\mathcal{E}}_k(\u(t))\vert=\fra{1}{2}\Vert\u(t)\Vert_{{\bf L}^2}^2\:$ and   $\Vert\u(t)\Vert_{{\bf L}^2}\le K$, we have that
\[
\vert {\mathcal{E}}_k(\u(t))\vert^{1-\theta}=\fra{1}{2^{1-\theta}}\Vert\u(t)\Vert_{{\bf L}^2}^{2(1-\theta)}
=
\fra{1}{2^{1-\theta}}\Vert\u(t)\Vert_{{\bf L}^2}^{1-2\theta}\Vert\u(t)\Vert_{{\bf L}^2}
\leq C\Vert\u(t)\Vert_{{\bf L}^2}\quad a.e.~t\ge 0.
\]
This estimate together  the Lojasiewicz-Simon inequality 
\[ 
\vert {\mathcal{E}}_\mu(Q(t))- E_\infty \vert^{1-\theta}\leq C \Vert H\Vert_{\mathbb{H}^{-1}},  \quad\mbox{ a.e.}~t\ge  t_1 .
\] 
give
\[
\begin{array}{c}
( {\mathcal{E}}(u(t),Q(t))-E_\infty)^{1-\theta}\leq
\vert {\mathcal{E}}_k(\u(t)) \vert^{1-\theta}+\vert \displaystyle  {\mathcal{E}}_\mu(Q(t))- E_\infty \vert^{1-\theta}\\
[2mm]
\leq C (\Vert \u(t)\Vert_{{\bf L}^2}+
  \Vert H(t)\Vert_{\mathbb{H}^{-1}})^{} \quad \mbox{ a.e.}~t\ge  t_1.
\end{array}
\]
Therefore, 
\begin{equation}\label{interm-L-S}
\quad\:
( {\mathcal{E}}(u(t),Q(t))-E_\infty)^{\theta-1}(\Vert \u(t)\Vert_{{\bf L}^2}+  \Vert H(t)\Vert_{\mathbb{H}^{-1}})^{} \geq
C  
  \end{equation}
almost every where $t\ge  t_1$.
By applying (\ref{interm-L-S}) in (\ref{estim-1s}),
\[
 \fra{d}{dt} \left(({\mathcal{E}}(u(t),Q(t))-E_\infty)^\theta\right) +{C}\,{\theta}\:(\Vert \grad\u(t)\Vert_{\mathbb{L}^2}+
\Vert H(t)\Vert_{\mathbb{L}^2})\leq 0,\quad \mbox{ a.e.}~t\ge  t_1
\]
hence $(\ref{stabs})$ is proved.

Here,  the hypothesis ${\mathcal{E}}(u(t),Q(t))=\widetilde{\mathcal{E}}(t)$ for almost every $t$ is a key point. In particular, this hypothesis implies that the integral and differential versions of the energy law (\ref{energyeqint}) and (\ref{energy-eqs}) are satisfied by ${\mathcal{E}}(u(t),Q(t))$ a.e.\ in time. In fact, energy law (\ref{energy-eqs}), changing $\widetilde{\mathcal{E}}(t)$ by  ${\mathcal{E}}(u(t),Q(t))$, is the crucial hypothesis imposed in Remark 2.4 of \cite{PRS}. 

 Fixed any $t_2\in (t_1,+\infty)$, taking into account that $({\mathcal{E}}(\u({t_2}),Q({t_2}))-E_\infty)^\theta>0$ and, integrating $(\ref{stabs})$ into $[t_1, t_2]$
   we have
\begin{equation}\label{estim-2s}
\begin{array}{c}
\theta\, C\jnt_{t_1}^{t_2}(\Vert \grad\u(t)\Vert_{\mathbb{L}^2}+
\Vert H(t)\Vert_{\mathbb{L}^2}) dt
  \leq ({\mathcal{E}}(\u({t_1}),Q({t_1}))-E_\infty)^\theta.
  \end{array}
\end{equation}
From the equation  (\ref{doblecero-a}), by using the weak regularity $Q\in L^\infty((0,+\infty)\times \Om)$, then 
\[
\Vert\partial_t Q(t)\Vert_{\mathbb{H}^{-1}}\leq C(\Vert \grad\u(t)\Vert_{\mathbb{L}^2}+
\Vert H(t)\Vert_{\mathbb{L}^2}) 
\qquad a.e.~t\ge 0.
\]
By using this inequality in  (\ref{estim-2s}), then (\ref{stab2s}) is attained. 
 
 \medskip

\noindent{\bf Step 2:}  {\sl There exists a sufficiently large $n_0$ such that  $\Vert
Q(t)-Q_\infty\Vert_{\mathbb{L}^2}\leq\beta_1$ and $\vert {\mathcal{E}}_\mu(Q(t))-{\mathcal{E}}_\mu(Q_*)\vert\leq\beta_2$ for all $t\geq t_{n_0}$ ($\beta_1, \beta_2$  given in Lemma~\ref{le:L-S2})}.
\medskip

Since $Q(t_n)\rightarrow Q_\infty$ strongly  in $\mathbb{L}^2$ and ${\mathcal{E}}(\u(t_n),Q(t_n)) \searrow E_\infty={\mathcal{E}}_\mu(Q_\infty)$ in $\mathbb{R}$ (see (\ref{asenergys})), then  for any $\delta\in (0,\beta_1)$, there exists an integer $N(\delta)$ such that, for all $n\geq N(\delta)$,
\begin{equation}\label{epsis}
\begin{array}{c}
\Vert Q(t_n)-Q_\infty\Vert_{\mathbb{L}^2}\leq\delta\quad
\mbox{ and }\quad
\fra{1}{\theta}({\mathcal{E}}_\mu(Q(t_n))- E_\infty)^\theta\leq\delta.
\end{array}
\end{equation}
For each $n\geq N(\delta)$, we define 
\[
\overline{t}_n:=\sup\{t: t>t_n, \,\Vert Q(s)-Q_\infty \Vert_{\mathbb{L}^2}<\beta_1\quad\forall s\in [t_n,t)\}.
\]
It suffices to prove that  $\overline{t}_{n_0}=+\infty$ for some  $n_0$. Assume by contradiction that  $t_n<\overline{t}_n<+\infty$ for all $n$, hence $\Vert Q(\overline{t}_n)-Q_\infty \Vert_{\mathbb{L}^2}=\beta_1$ and $\Vert Q({t})-Q_\infty \Vert_{\mathbb{L}^2}<\beta_1$ for all $t\in [t_n, \overline t_n)$.   By applying Step 1 for all $t\in [t_n,\overline{t}_n]$, from $(\ref{stab2s})$ and $(\ref{epsis})$ we obtain, 
\[
\jnt_{t_{n}}^{\overline{t}_n}\Vert\partial_tQ\Vert_{\mathbb{H}^{-1}}\leq C\delta, \quad \forall\,n\geq N(\delta).
\]
Therefore, 
\[
\begin{array}{c}
\Vert Q(\overline{t}_n)-Q_\infty\Vert_{\mathbb{H}^{-1}}\leq
\Vert Q(t_n)-Q_\infty\Vert_{\mathbb{H}^{-1}}
+\jnt_{t_{n}}^{\overline{t}_n}\Vert\partial_tQ\Vert_{\mathbb{H}^{-1}}\leq (1+C)\delta,
\end{array}
\]
which implies that $\lim_{n\rightarrow +\infty}\Vert Q(\overline{t}_n)-Q_\infty\Vert_{\mathbb{H}^{-1}}=0$. 

On the other hand, $ Q(\overline{t}_n)$ is bounded in ${\mathbb{H}^1}$. Indeed, from (\ref{asenergys}), $\widetilde{\mathcal{E}}(\u(\overline{t}_n),Q(\overline{t}_n))$  is bounded in $\mathbb{R}$, therefore in particular 
 \[
 \jnt_\Omega{\mathcal{E}}_\mu(Q(\overline{t}_n))\,dx
 =\jnt\displaystyle\Big(\frac{\varepsilon}{2}\vert
\nabla Q (\overline{t}_n)\vert^2 + F_\mu(Q(\overline{t}_n)) \Big)\,dx
\]
 is bounded.
But, since $ F_\mu(Q)$ is bounded in $L^\infty(\mathbb{L}^1) $, then $\grad Q(\overline{t}_n)$ is bounded in $\mathbb{L}^2(\Om)$, therefore $ Q(\overline{t}_n)$ is bounded in $\mathbb{H}^1$.
  
  Consequently, $ Q(\overline{t}_n)$    is relatively compact in ${\mathbb{L}^2}$, hence there exists a subsequence of $ Q(\overline{t}_n)$, which is still denoted as $ Q(\overline{t}_n)$, that converges to $Q_\infty$ in ${\mathbb{L}^2}$-strong. Hence  $\Vert Q(\overline{t}_n)-Q_\infty\Vert_{{\mathbb{L}^2}}<\beta_1$ for a sufficiently large $n$, which contradicts the definition of $\overline{t}_n$. 
\medskip

\noindent{\bf Step 3:} {\sl There exists a unique $Q_\infty$ such that $Q(t) \to Q_\infty$ weakly in $\mathbb{H}^{1}$ as $t\uparrow +\infty$.}
\medskip

By using Steps 1 and 2,  (\ref{stab2s})  can be applied, for all $t_1,t_0: t_1>t_0\geq t_{n_0}$, hence
\[
\Vert Q(t_1)-Q(t_0)\Vert_{\mathbb{H}^{-1}}\leq
\jnt_{t_0}^{t_1}\Vert\partial_tQ\Vert_{\mathbb{H}^{-1}} \to 0,\:\: \hbox{as $t_0,t_1\to +\infty$.}
\]
Therefore, 
$(Q(t))_{t\geq {t_{n_0}}}$ is a Cauchy sequence in $\mathbb{H}^{-1}$ as $t\uparrow +\infty$, hence, 
there exists a unique $Q_\infty\in \mathbb{H}^{-1}$ such that $Q(t)\to Q_\infty$ in $\mathbb{H}^{-1}$ as $t\uparrow +\infty$. Finally, the   convergence in $\mathbb{H}^1$-weak by sequences of $Q(t)$ proved in Theorem~\ref{the:first}, yields  to $Q(t)\to Q_\infty$ in $\mathbb{H}^1$-weak, and the proof is finished.

\end{document}